\documentclass[letter,12pt]{amsart}

\usepackage{tikz}
\usepackage[all]{xy}
\usetikzlibrary{decorations.markings}
\usetikzlibrary{shapes}
\usetikzlibrary{backgrounds}
\usepackage{subfig}
\usepackage{amssymb} 

\usepackage{hyperref}
\def\thetitle{{Right-angled Artin groups and finite subgraphs of curve graphs}}
\hypersetup{
    pdftitle=   \thetitle,
   pdfauthor=  {Sang-hyun Kim and Thomas Koberda}
}
\usepackage{enumerate}
\makeatletter
\let\@@enum@org\@@enum@
\def\@@enum@[#1]{\@@enum@org[\normalfont #1]}
\makeatother

\newtheorem{thm}{Theorem}
\newtheorem{lem}[thm]{Lemma}

\newtheorem{prop}[thm]{Proposition}

\newtheorem{que}{Question}

\theoremstyle{remark}

\newtheorem*{rem}{Remark}

\theoremstyle{definition}
\newtheorem{defn}[thm]{Definition}

\newcommand\form[1]{\langle #1\rangle}

\newcommand\co{\colon}

\newcommand\Z{\mathbb{Z}}

\newcommand\supp{\operatorname{supp}}

\newcommand\lk{\operatorname{Lk}}
\newcommand\st{\operatorname{St}}

\newcommand\gex{{\Gamma}^e}
\newcommand\aga{{A(\Gamma)}}

\newcommand\gam{\Gamma}
\newcommand\bZ{\mathbb{Z}}
\newcommand\Mod{\operatorname{Mod}}

\newcommand\mC{\mathcal{C}}

\begin{document}

\title\thetitle

\author{Sang-hyun Kim}
\address{Department of Mathematical Sciences, KAIST, 335 Gwahangno, Yuseong-gu, Daejeon 305-701, Republic of Korea}
\email{shkim@kaist.edu}

\author{Thomas Koberda}
\address{Department of Mathematics, Yale University, 20 Hillhouse Ave, New Haven, CT 06520, USA}
\email{thomas.koberda@gmail.com}
\date{\today}
\begin{abstract}
We show that for a sufficiently simple surface $S$, a right-angled Artin group $A(\Gamma)$ embeds into $\Mod(S)$ if and only if $\Gamma$ embeds into the curve graph $\mC(S)$ as an induced subgraph. When $S$ is sufficiently complicated, there exists an embedding $A(\Gamma)\to\Mod(S)$ for some $\Gamma$ not contained in $\mC(S)$.
\end{abstract}
\maketitle
\section{Introduction}
\subsection{Statement of the main results}
Let $S=S_{g,n}$ be a connected orientable surface of genus $g$ and $n$ punctures, and let $\Mod(S)$ denote its mapping class group.  As is standard, we will write \[\xi(S)=\max(3g-3+n,0)\] for the \emph{complexity} of $S$, a measure which coincides with the number of components of a maximal multicurve on $S$.
We will use $\mC(S)$ to denote the \emph{curve graph} of $S$, which is the $1$--skeleton of the \emph{curve complex} of $S$.  The vertices of $\mC(S)$ are isotopy classes of essential, nonperipheral, simple closed curves on $S$.  Two vertices are adjacent if the corresponding isotopy classes admit disjoint representatives.

Let $\gam$ be a finite simplicial graph with vertex set $V(\gam)$ and edge set $E(\gam)$.  We will write $\aga$ for the \emph{right-angled Artin group} on $\gam$, which is defined by \[\aga=\langle V(\gam)\mid [v_i,v_j]=1 \textrm{ if and only if } \{v_i,v_j\}\in E(\gam)\rangle.\]

For two groups $G$ and $H$, we will write $G\le H$ if there is an embedding from $G$ into $H$.
Similarly, we write $\Lambda\le \Gamma$ for two graphs $\Lambda$ and $\Gamma$ if $\Lambda$ is isomorphic to an induced subgraph of $\Gamma$.
The central question of this article is exactly when the converse of the following result of the second author is true: 

\begin{thm}[\cite{Koberda2012}, Theorem 1.1 and Proposition 7.16]\label{t:kob}
If $\Gamma\le \mC(S)$ then $A(\Gamma)\le \Mod(S)$.
\end{thm}

We first prove the following results.

\begin{thm}\label{t:main1}
Let $S$ be a surface with $\xi(S)<3$.
Then $A(\Gamma)\le\Mod(S)$ if and only if $\Gamma\le\mC(S)$.
\end{thm}

\begin{thm}\label{t:main2}
Let $S$ be a surface with $\xi(S)>3$.
Then there exists a finite simplicial graph $\Gamma$ such that $A(\Gamma)\le \Mod(S)$ but $\Gamma\not\le\mC(S)$.
\end{thm}

A \emph{clique} is a subset of the vertex set which spans a complete subgraph.
A \emph{facet} of a triangulation of a manifold means a top-dimensional simplex.
For a positive integer $N$, we will say that $\gam$ has \emph{$N$--thick stars} if each vertex $v$ of $\gam$ is contained in two cliques $K_1\cong K_2$ on $N$ vertices of $\gam$ whose intersection is exactly $v$.  Equivalently, the link $\lk(v)$ of $v$ in $\gam$ contains two disjoint copies of complete graphs on $N-1$ vertices.  
For example, a \emph{proper} (namely, no two facets share more than one faces) triangulation of a compact surface with no triangular links has $3$-thick stars. The following generalization is immediate.

\begin{prop}
A proper triangulation of a compact $(N-1)$--manifold has $N$-thick stars if and only if the link of each vertex has at least $N+1$ facets.
\end{prop}

Having $N$-thick stars forces the converse of Theorem~\ref{t:kob} to be true.

\begin{thm}\label{t:nthick}
Suppose $S$ is a surface with $\xi(S)=N$ and $\gam$ is a finite graph with $N$--thick stars. If $\aga\le \Mod(S)$, then $\gam\le\mC(S)$.
\end{thm}

The methods developed in this paper are ineffective in the case $\xi(S)=3$, and the authors do not know which conclusion of Theorems \ref{t:main1} or \ref{t:main2} holds in that case.  There are exactly three surfaces with $\xi(S)=3$, though there are only two different cases to consider among them (see Section \ref{s:intermediate}).

\begin{que}
Let $S$ be a surface of complexity $3$.  Do there exist subgroups $\aga<\Mod(S)$ such that $\gam\not\le\mC(S)$?
\end{que}

\subsection{Notes and references}
Aside from Theorem~\ref{t:kob}, the second main motivation for Theorems \ref{t:main1} and \ref{t:main2} is the following classical result of Birman, Lubotzky and McCarthy:

\begin{thm}[\cite{BLM1983}]\label{t:BLM1983}
Let $S$ be a surface and $A$ be a torsion-free abelian subgroup of $\Mod(S)$.
Then $A$ is isomorphic to a subgroup $\form{\phi_1,\ldots,\phi_k}<\Mod(S)$ where each $\phi_i$ is a Dehn twist or a pseudo-Anosov map on a connected subsurface
and $\supp\phi_i\cap \supp \phi_j=\varnothing$ for $i\ne j$.
\end{thm}

Let $K_n$ denote a complete graph on $n$ vertices.  In our language, Theorem \ref{t:BLM1983} implies that $A(K_n)\le \Mod(S)$ if and only if $K_n\le\mC(S)$.

The relationship between right-angled Artin groups and mapping class groups of surfaces has been studied by many authors from various perspectives (see \cite{CP2001}, \cite{CW2004}, \cite{CLM2012}, \cite{Koberda2012}, \cite{KK2013a} and the references therein, for instance).  Our perspective stems from the following theorem, which can be obtained by combining a result of the authors with a result of the second author (see \cite{Koberda2012} and \cite{KK2013} or \cite{KK2013a}):

\begin{thm}[See \cite{acyl}]\label{t:raagmcg}
Let $\gam$ be a finite graph and let $S$ be a surface.
\begin{enumerate}
\item
Let $i$ be an embedding of $\gam$ into $\mC(S)$ as an induced subgraph. 
Then for all sufficiently large $N$, the map \[i_{*,N}:\aga\to\Mod(S)\] given by sending $v$ to the $N^{th}$ power of a Dehn twist $T_{i(v)}^N$ is injective.
\item
If $\aga\le\Mod(S)$, then there is an embedding of $\gam$ into $\mC(S)_k$ as an induced subgraph.
\end{enumerate}
\end{thm}

Observe that the first part of Theorem~\ref{t:raagmcg} is a more precise version of Theorem~\ref{t:kob}.  Here, the graph $\mC(S)_k$ denotes the \emph{clique graph} of $\mC(S)$.  From a topological perspective, $\mC(S)_k$ can be defined as the graph whose vertices are isotopy classes of essential, non--peripheral multicurves on $S$, and where two vertices are adjacent if the corresponding multicurves are component--wise parallel or disjoint.  Theorem \ref{t:main2} shows that $\mC(S)_k$ in Theorem \ref{t:raagmcg} cannot be replaced by $\mC(S)$ for a general surface.

In \cite{KK2013a}, in \cite{KK2013} and in \cite{acyl}, the authors develop an analogous theory of curve graphs for right-angled Artin groups.  In particular, a verbatim analogue of Theorem \ref{t:raagmcg} holds with $\Mod(S)$ replaced by a right-angled Artin group A($\Lambda$) and the curve graph $\mC(S)$ replaced by the extension graph $\Lambda^e$ of $\Lambda$.  
For many classes of graphs, it is known that $\aga\le A(\Lambda)$ if and only if $\gam\le\Lambda^e$, for instance for triangle--free graphs~\cite{KK2013} and for $C_4$ and $P_3$--free graphs (so--called \emph{thin--chordal graphs})~\cite{CDK2013}. 
However Casals--Ruiz, Duncan and Kazachkov proved that this is not always the case.

\begin{thm}{\cite{CDK2013}}\label{t:cdk}
There exist finite graphs $\gam$ and $\Lambda$ such that $\aga<A(\Lambda)$ and such that $\gam\not\le\Lambda^e$.
\end{thm}

Thus, Theorem \ref{t:main2} can be viewed as an analogue of Theorem \ref{t:cdk} for mapping class groups.  We note briefly that Theorem \ref{t:cdk} does not imply Theorem \ref{t:main2}, for even if a particular graph $\gam$ embeds in $\mC(S)_k$, the graph $\mC(S)_k$ is vastly more complicated that $\gex_k$. However, our example in Section~\ref{s:high comp} gives another example of $A(\Lambda)\le A(\Gamma)$ but $\Lambda\not\le\Gamma^e$; see the remark following Lemma~\ref{lem:gam0}.


The concept of $N$--thick stars used in Theorem \ref{t:nthick} is related to the well--studied graph--theoretic notion of a \emph{quasi--line} (see \cite{chudseymour}, for instance).  A graph is a quasi--line if the star of each vertex is the union of two complete graphs.

\section{Acknowledgements}
The authors thank M. Clay for asking the question which led to this article, and for useful conversations.  The authors also thank A. Lubotzky, C. McMullen, Y. Minsky and S. Oum for useful discussions.  The authors thank an anonymous referee for helpful comments and corrections.
The second named author is partially supported by NSF grant DMS-1203964.

\section{Proof of Theorem~\ref{t:main1}}
Let $S$ be a surface with punctures.
A mapping class $\phi\in\Mod(S)$ is called a \emph{multi-twist} if $\phi$ can be represented by 
a multiplication of powers of Dehn twists along disjoint pairwise-non-isotopic simple closed curves.
We call a regular neighborhood of the union of those simple closed curves as the \emph{support} of $\phi$.
The following is a refinement of a lemma in~\cite{KK2013a}, and the proof is similar.


\begin{lem}\label{lem:standard}
Let $X$ be a finite graph.
If $A(X)\le \Mod(S)$ then there exists an embedding $f\co A(X)\to \Mod(S)$ satisfying the following:
\begin{enumerate}[(i)]
\item
The map $f$ maps each vertex of $X$ to a multi-twist;
\item
For two distinct vertices $u$ and $v$ of $X$, the support of $f(u)$ is not contained in the support of $f(v)$.
\end{enumerate}
\end{lem}

\begin{proof}
Let $f_0$ be an embedding of $A(X)$ into $\Mod(S)$.
By raising the generators to powers if necessary, we may assume that the image of each vertex $v$ is written as $\phi_1^v \phi_2^v\cdots \phi_{n_v}^v$
where each $\phi_i^v$ is either a Dehn twist or a pseudo-Anosov on a connected subsurface
and $\phi_i^v$'s have disjoint supports.
Choose a minimal collection $\{\psi_1,\ldots,\psi_m\}\subseteq\Mod(S)$ such that for every $i$ and $v$,
the mapping class $\phi_i^v$ is a power of some $\psi_j$.
By~\cite{Koberda2012} (see also~\cite{CLM2012}),
there exists $N>0$ and a graph $Y$ with $V(Y)=\{v_1,\ldots,v_m\}$
such that the map $g_0\co A(Y)\to \Mod(S)$ defined by $g(v_j) = \psi_j^N$ 
is an embedding.
Moreover,
we can find simple closed curves $\gamma_1,\ldots\gamma_m$ such that 
$\gamma_i\subseteq \supp\psi_i$ 
and $\supp \psi_i\cap\supp\psi_j=\varnothing$ if and only if $\gamma_i\cap \gamma_j=\varnothing$
for every $i$ and $j$.
By raising $N$ further if necessary, we have an embedding $g\co A(Y)\to \Mod(S)$ defined by 
$v_j\mapsto T_{\gamma_j}^N$.
We may assume that $f_0 = g_0\circ h$ for some $h\co A(X)\to A(Y)$,
by further raising the image of each $f(v)$ for $v\in X$ to some power.
Then $g\circ h$ is an embedding from $A(X)$ to $\Mod(S)$ such that each vertex maps to a multi-twist.
Note that if $u$ and $v$ are adjacent vertices in $X$ then the multi-curves corresponding to $f(u)$ and $f(v)$ also form a multi-curve.

Now among the embeddings $f\co A(X)\to\Mod(S)$ that map each vertex to a multi-twist,
we choose $f$ so that \[\sum_v \#\supp f(v)\] is minimal.
Here, $\#$ of a multi-curve denotes the number of components.
Suppose that $\supp f(v)\subseteq \supp f(w)$.
Note that $[v,w]=1$ and $\lk_X(w)\subseteq \st_X(v)$.
Then we have a \emph{transvection}  (see \cite{Servatius1989})  automorphism $\tau\co A(X)\to A(X)$ defined by $\tau(w)=v^P w^Q$
for some $P,Q\ne0$ such that 
$\#\supp f(v^P w^Q)<\#\supp f(v)$, a fact which results more or less from the Euclidean algorithm.  The reader may consult \cite{KK2013a} for more details. This is a contradiction of the minimality of $\sum_v \#\supp f(v)$.
\end{proof}

\begin{rem}\label{rem:standard}
In the above lemma, if $\supp f(v)$ is a maximal clique in $\mC(S)$ then the condition (ii) implies that $v$ is an isolated vertex.
\end{rem}

\begin{defn}\label{defn:standard}
An embedding of a right-angled Artin group into a mapping class group is called \emph{standard}
if conditions (i) and (ii) in Lemma~\ref{lem:standard} are satisfied.
\end{defn}

The lowest complexity surfaces with nontrivial mapping class groups are $S_{0,4}$ and $S_{1,1}$ (so that $\xi(S)=1$).  Both of these surfaces admit simple closed curves, but neither admits a pair of disjoint isotopy classes of simple closed curves.  Because of this fact, most authors define edges in $\mC(S)$ to lie between curves with minimal intersection (two or one intersection point, respectively).  This definition is not suitable for our purposes and we will keep the standard definition of curve graphs, so that $\mC(S)$ is an infinite union of isolated vertices in both of these cases.  

\begin{proof}[Proof of Theorem~\ref{t:main1}]
In view of Theorem \ref{t:raagmcg}, we only need to show the ``only if" direction.  
The case $\xi(S)=1$ is obvious, since $\mC(S)$ is discrete and $\Mod(S)$ is virtually free.
Now let us assume $\xi(S)=2$, so that $S=S_{1,2}$ or $S=S_{0,5}$.  
We note that $\mC(S)$ contains no triangles.  

The conclusion of the proposition holds for $\gam$ if and only if it holds for each component of $\gam$. This is an easy consequence of the fact that $\mC(S)$ has infinite diameter and that a pseudo-Anosov mapping class on $S$ exists. So, we may suppose that $\gam$ is connected and contains at least one edge.
By Lemma~\ref{lem:standard}, we can further assume to have a standard embedding $f\co A(\gam)\to \Mod(S)$.
Since $\gam$ has no isolated vertices and $\mC(S)$ is triangle--free,
the remark following Lemma~\ref{lem:standard} implies that
each vertex maps to a power of a single Dehn twist.
This gives a desired embedding $\gam\to\mC(S)$.
\end{proof}

\section{High complexity surfaces}\label{s:high comp}
The strategy for dealing with high complexity surfaces (surfaces $S$ for which $\xi(S)>3$) is to build an example which works for surfaces with $\xi(S)=4$ and then bootstrapping to obtain examples in all higher complexities.  In particular, we will take the three surfaces with $\xi(S)=4$ and build graphs $\gam_0$ and $\gam_1$ such that $A(\gam_0)<A(\gam_1)<\Mod(S)$ but such that $\gam_0\nleq\mC(S)$.  We will then use $\gam_0$ and $\gam_1$ to build corresponding graphs for surfaces of complexity greater than four.

The source of our examples in this section will be the graphs $\gam_0$ and $\gam_1$ shown in Figure~\ref{fig:gam}. Observe that the graph $\Gamma_0$ is obtained from the graph $\gam_1$ by collapsing $e$ and $f$ to a single vertex $q$ and retaining all common adjacency relations.  We will denote by $C_4$ the 4-cycle spanned by $\{a,b,c,d\}$.

\begin{figure}[htb!]
\subfloat[(a) $\Gamma_0$]{\includegraphics[width=.3\textwidth]{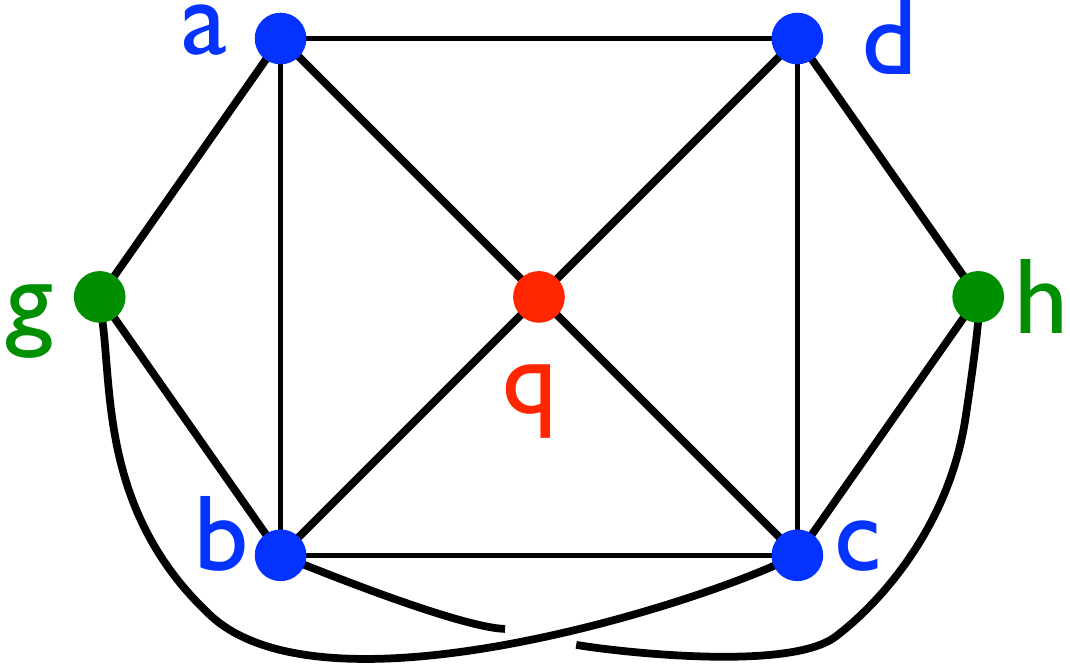}}
$\qquad$$\qquad$
\subfloat[(b) $\Gamma_1$]{\includegraphics[width=.3\textwidth]{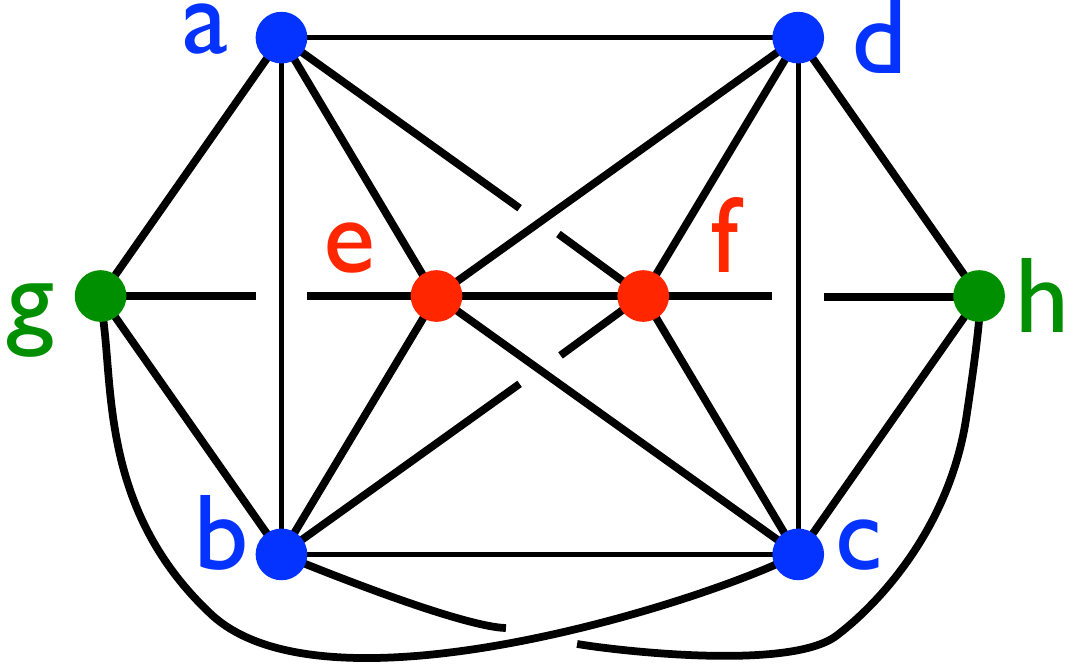}}
  \caption{Two graphs $\Gamma_0$ and $\Gamma_1$.}
  \label{fig:gam}
\end{figure}


\subsection{An algebraic lemma}
Let us consider the map $\phi:A(\gam_0)\to A(\gam_1)$ defined by $\phi:q\mapsto ef$ and which is the identity on the remaining vertices. 

\begin{lem}\label{l:alaginj}
The map $\phi:A(\gam_0)\to A(\gam_1)$ is injective.
\end{lem}
\begin{proof}
We first claim that the restriction $\psi\co \form{C_4,q,h}\to\form{C_4,ef,h}$ of $\phi$ is an isomorphism.
To see this, consider the projection $p\co \form{C_4,e,f,h}\to\form{C_4,e,h}$ defined by $p(f)=1$.
The  claim follows from that $p\circ\psi$ is an isomorphism.


Now suppose $w$ is a reduced word in $\ker\phi\setminus\{1\}$.
By choosing an innermost $\{g,g^{-1}\}$--pair in the cancellation diagram (see~\cite{CW2004,Kim2008}) of the word $\phi(w)$, we may write $w = w_0 g^{\pm1} w_1 g^{\mp1} w_2$
such that $w_1\in\form{C_4,q,h}\cap\phi^{-1}\form{a,b,c,e}$.  Indeed, since the $\{g,g^{-1}\}$--pair cancels, we must have that $\supp(\phi(w_1))\subset\st(g)$.
Then we have 
\[\phi(w_1)\in\phi\form{C_4,q,h}\cap\form{a,b,c,e}=\form{C_4,ef,h}\cap(\form{a,c}\times\form{b}\times\form{e}).\]
Since $\phi(w_1)\in\form{a,b,c,e}$, the exponent sum of $f$ in $\phi(w_1)$ is zero.
From $\phi(w_1)\in\form{C_4,ef,h}$, it follows that the exponent sum of $e$ in $\phi(w_1)$ is also zero.
Since $\form{e}$ is a direct factor of $\form{a,b,c,e}$, we see $\phi(w_1)\in\form{a,b,c}$. Combined with the claim in the first paragraph, we have
\[
w_1\in\phi^{-1}\form{a,b,c}\cap\form{C_4,q,h}=\psi^{-1}\form{a,b,c}=\form{a,b,c}.\]
This contradicts the fact that $w$ is reduced.
\end{proof}

\subsection{The case $\xi(S)=4$}
Let $S$ be a connected surface with complexity four.
This means  $S$ is one of $S_{0,7}$, $S_{1,4}$ and  $S_{2,1}$.

\begin{lem}\label{l:gamma0embed}
The graph $\gam_1$ embeds into $\mC(S)$ as an induced subgraph.
\end{lem}

\begin{proof}
The corresponding surfaces are shown in Figure~\ref{fig:xi4}. In (a) and (c), the curves for the vertices $a,c,e$ and $h$ are given by the mirror images of those for $b,d,f$ and $g$, respectively. 
One can verify that the curves with the shown configuration have the minimal intersections by observing that the intersection numbers are either 0,1 or 2.
\end{proof}

\begin{figure}[htb!]
\subfloat[(a) $S_{0,7}$]{\includegraphics[width=.27\textwidth]{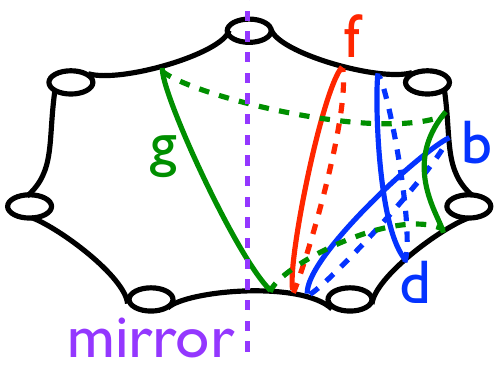}}
$\qquad$
\subfloat[(b) $S_{1,4}$]{\includegraphics[width=.3\textwidth]{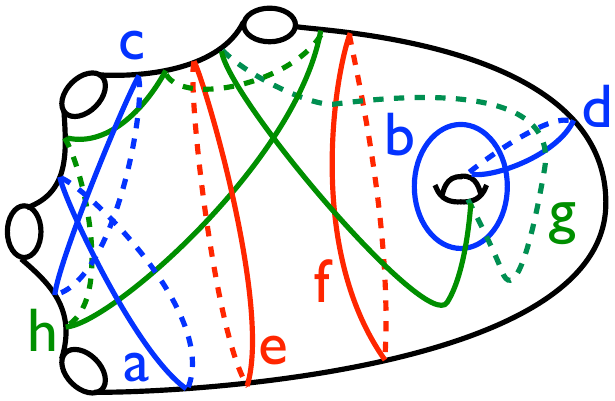}}
$\qquad$
\subfloat[(c) $S_{2,1}$]{\includegraphics[width=.21\textwidth]{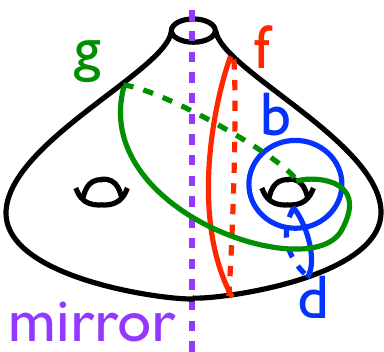}}
  \caption{Complexity four surfaces.}
  \label{fig:xi4}
\end{figure}

Now suppose $\{a,b,c,d\}$ are simple closed curves on $S$ (here we still have $\xi(S)=4$) which form a four-cycle in $\mC(S)$ with this cyclic order.
Let $S_1$ be a closed regular neighborhood of the curves $a$ and $c$ along with disks glued to null-homotopic boundary components.
Similarly we define $S_2$ for $b$ and $d$ so that $S_1\cap S_2=\varnothing$. Define $S_0$ as the closure of $S\setminus (S_1\cup S_2)$.  
By isotopically enlarging $S_1$ and $S_2$ if necessary, we may assume that whenever $A$ is an annulus component of $S_0$ then both components of $\partial A$ intersect $S_1\cup S_2$
(i.e. no component of $S_0$ is a punctured disk).
Note that $\xi(S_1),\xi(S_2)\ge1$, since they both contain a pair of non--isotopic simple closed curves. Since $S$ is connected and at least one component of $S_0$ intersects each of $S_1$ and $S_2$, we have that $S_0$ has at least two boundary components.

\begin{lem}\label{lem:c4}
The triple $(S_0,S_1,S_2)$ satisfies exactly one of the following conditions, possibly after switching the roles of $S_1$ and $S_2$.
\begin{enumerate}[(i)]
\item
$S_1\in\{S_{1,2},S_{0,5}\}, S_2\in\{S_{1,1},S_{0,4}\}$, $S_0\approx S_{0,2}$,
and $S_0$ intersects both $S_1$ and $S_2$.
\item
$S_1,S_2\in\{S_{0,4},S_{1,1}\}$, $S_0\approx S_{0,3}$,
and $S_0$ intersects each of $S_1$ and $S_2$ at only one boundary component.
\item
$S_1,S_2\in\{S_{0,4}\}$, $S_0\approx S_{0,2}\coprod S_{0,2}$,
and each component of $S_0$ intersects both $S_1$ and $S_2$.
\item
$(S_1,S_2)\in\{(S_{0,4},S_{0,4}),(S_{0,4},S_{1,1})\}$, $S_0\approx S_{0,2}\coprod S_{0,2}$,
and one component of $S_0$ intersects each of $S_1$ and $S_2$ at only one boundary component,
while the other component of $S_0$ intersects $S_1$ at two boundary components.
\item
$(S_1,S_2)\in\{(S_{0,4},S_{0,4}),(S_{0,4},S_{1,1})\}$ and $S_0\approx S_{0,2}\coprod S_{0,3}$
such that
the $S_{0,2}$ component intersects both $S_1$ and $S_2$
and the $S_{0,3}$ component is disjoint from $S_2$, and moreover, $S_{0,3}\cap S_1\approx S^1$.
\end{enumerate}
\end{lem}

\begin{proof}
Let $\alpha$ be the number of free isotopy classes of boundary components of $S_0$ that are contained in $S_1\cup S_2$.
We have $\alpha>0$ since $S$ is connected and $S_1\cap S_2=\varnothing$.
Then $\xi(S) = \xi(S_1)+\xi(S_2)+\xi(S_0)+\alpha$; here, $\xi(S_0)$ is defined as the sum of the complexities of the components of $S_0$~\cite{BLM1983}.
It follows that $2\le \xi(S_1)+\xi(S_2)\le3$.

Let us first assume $\xi(S_1)+\xi(S_2)=3$. From $\xi(S_0)+\alpha=1$, we see that $S_0$ is an annulus joining $S_1$ and $S_2$. Case (i) is immediate. 

Now we assume $\xi(S_1)=\xi(S_2)=1$. If $\alpha=1$, then $S_0$ is forced to be an annulus and we have a contradiction of the fact that $\xi(S_0)+\alpha=2$. So we have $\xi(S_0)=0$ and $\alpha=2$.
If $S_0$ is connected, then $\alpha=2$ implies that $S_0$ cannot be an annulus, and hence, Case (ii) follows.
So we may assume $S_0$ is disconnected. 

Suppose $S_0\approx S_{0,2}\coprod S_{0,2}$. If each component of $S_0$ intersects both of $S_1$ and $S_2$, 
then each $S_i$ has at least two boundary components for $i=1,2$. In particular, $S_i\ne S_{1,1}$ and Case (iii) follows.
Without loss of generality, let us assume that one component of $S_0$ intersects only $S_1$. Then $S_1\ne S_{1,1}$
and we have Case (iv).

Let us finally assume $S_0\approx S_{0,2}\coprod S_{0,3}$. This is the only remaining case, for $\alpha=2$.
The subsurface $S_{0,3}$ can contribute only one to $\alpha$. In particular, $S_{0,3}$ is glued to say, $S_1$ but not $S_2$.
The annulus component of $S_0$ joins $S_1$ and $S_2$ and therefore, $S_1\ne S_{1,1}$ and Case (v) follows.
\end{proof}

The following special case of Theorem~\ref{t:main2} will be central to our discussion of surfaces with $\xi(S)\geq 4$:

\begin{lem}\label{lem:gam0}
Let $S$ be a surface with $\xi(S)=4$.  
There exists an embedding from $A(\gam_0)$ into $\Mod(S)$, but $\gam_0$ does not embed into $\mC(S)$ as an induced subgraph.
\end{lem}

\begin{proof}
The first half of the conclusion follows from Lemmas \ref{l:alaginj} and \ref{l:gamma0embed}, combined with Theorem \ref{t:raagmcg} (1). For the second half, let us assume $\gam_0\le \mC(S)$ and regard the vertices $a,b,c,\ldots$ as simple closed curves on $S$. From $C_4\le \gam_0$, we have one of the five cases in Lemma~\ref{lem:c4}.
From the adjacency relations in $\gam_0$, we observe that
$q\cap g, q\cap h, g\cap S_2,h\cap S_1$ and $g\cap h$ are all non-empty,
and also that $q\subseteq S_0$ and $g \cap S_1 = h\cap S_2=\varnothing$.

In Case (i), the annulus $S_0$ connects $S_1$ and $S_2$. This implies that $g\subseteq S_2, h\subseteq S_1$ and so, $g\cap h=\varnothing$. This is a contradiction. In Case (iii) and (iv), we similarly obtain a contradiction from $g\cap h=\varnothing$.

In Case (ii), the curve $q$ must be boundary parallel in $S_0$. Hence, it must be either $S_0\cap S_1$ or $S_0\cap S_2$. By symmetry, we may assume $q = S_0\cap S_1$. Then $q$ separates $S_1$ from $S$, and so, $g\cap q \subseteq g\cap S_0 = \varnothing$. This is a contradiction.
The proof for Case (v) is similar and goes as follows. The subsurface $S_1$ separates $S_2$ and $S_{0,3}\subseteq S_0$. This forces $g\subseteq S_2$, so that $g\cap q \subseteq g\cap S_0 =\varnothing$.
\end{proof}

\begin{rem}\label{r:ext}
Since $\Gamma_1\le\mC(S)$, we have $\Gamma_1^e\le\mC(S)$ by~\cite{KK2013a}. Hence we have another example of graphs $\Gamma_0\not\le\Gamma_1^e$
but  $A(\Gamma_0)\le A(\Gamma_1)$; see~\cite{KK2013,CDK2013}.
\end{rem}

\subsection{Surfaces with complexity larger than four}
For a graph $X$, let us define $\eta(X)$ to be the minimum of $\xi(S)$ among connected  surfaces $S$ satisfying $X\le\mC(S)$.
Note that $\eta(X)$ is at least the size of a maximal clique in $X$. 
So far we proved that $\eta(\gam_0)>4$.
A graph is \emph{anti-connected} if its complement graph ${V(X)\choose 2}\setminus X$ is connected.  Note that the graphs $\gam_0$ and $\gam_1$ are both anti--connected.

\begin{lem}\label{l:eta join}
If $X$ is a finite anti-connected graph and $n\ge0$,
then $\eta(X\ast K_n) \ge \eta(X)+n$.
 \end{lem}

\begin{proof}
Choose a surface $S$ such that
 $\xi(S)=\eta(X\ast K_n)$ and $X\ast K_n\le \mC(S)$.
Let $N$ denote a regular neighborhood of curves in $K_n$.
Since the graph $X$ is anti-connected, the curves in $V(X)$ are contained in a component $S_1$ of $S\setminus N$.
Since $X\le \mC(S_1)$, we have $\xi(S)=\xi(S\setminus N)+n \ge \xi(S_1)+n\ge \eta(X)+n$. 
\end{proof}

Put $\Lambda_n=\gam_0\ast K_{n-4}$ for $n\ge4$. Theorem~\ref{t:main2} is an immediate consequence of the following.
\begin{prop}\label{p:lambdan}
If $S$ is a surface with $\xi(S)=n$, then $A(\Lambda_n)$ embeds into $\Mod(S)$ but
$\Lambda_n$ is not an induced subgraph of $\mC(S)$.
\end{prop}
\begin{proof}
Choose a multicurve $X$ on $S$ with $n-4$ components such that $S\setminus X$ has a connected component $S_0$ of complexity at least four.  We have that $\mC(S_0)$ contains a copy of $\gam_1$, so that $A(\Gamma_1)\times\bZ^{n-4}$ embeds in $\Mod(S)$.  It follows that $A(\Gamma_0)\times\bZ^{n-4}\cong A(\Lambda_n)$ embeds in $\Mod(S)$.
On the other hand, Lemma~\ref{l:eta join} implies that
$\eta(\Lambda_n) 
\ge \eta(\Gamma_0) +n-4 >n$.
\end{proof}

\section{Proof of Theorem~\ref{t:nthick}}
In this section, we give a proof of Theorem \ref{t:nthick}.  
For a multi-curve $A$ on a surface $S$, we denote by $\form{A}$ the subgroup of $\Mod(S)$ generated by the Dehn twist about the curves in $A$.

\begin{proof}[Proof of Theorem \ref{t:nthick}]
By Lemma~\ref{lem:standard}, there exists a standard embedding $\phi:\aga\to\Mod(S)$. 
Let $v$ be an arbitrary vertex of $\gam$.  Write $K$ and $L$ for two disjoint cliques of $\gam$ such that
$K\coprod\{v\}$ and $L\coprod\{v\}$ are cliques on $N$ vertices.
The support of $\phi\form{K}$ is a multi-curve, say $A$.
Similarly we write $B=\supp\phi\form{L}$ and $C=\supp\phi\form{v}$.
Since $\xi(S)=N$, the multi-curves $A\cup C$ and $B\cup C$ are maximal. 
Note that $\form{C}$ is a subgroup of $\form{A\cup C}\cap \form{B\cup C}$.
In the diagram below, 
we see that $\phi\form{v}$ is of finite-index in $\form{C}\cong\Z^{|C|}$
and hence, $|C|=1$.
 \end{proof}

\begin{figure}[htb!]
\[
\xymatrix{ 
&\phi(A(\Gamma))\le\Mod(S)\ar@{-}[dl]\ar@{-}[dr]&\\
\form{A\cup C}\cong\Z^N\ar@{-}[d]^{\text{finite-index}}\ar@{-}[dr]
&&\form{B\cup C}\cong\Z^N\ar@{-}[d]^{\text{finite-index}}\ar@{-}[dl] \\
\phi\form{K,v}\cong\Z^N\ar@{-}[dr]& \form{C}\ar@{-}[d] &\phi\form{L,v}\cong\Z^N\ar@{-}[dl]\\
&\phi(\form{K,v}\cap\form{L,v})=\phi\form{v}\cong\Z
}
\]
  \label{fig:thick proof}
\end{figure}

\section{Remarks on intermediate complexity surfaces}\label{s:intermediate}
There are only three surfaces of complexity three: $S_{2,0}$, $S_{1,3}$ and $S_{0,6}$.  From the perspective of Theorem \ref{t:main2}, these three surfaces collapse into at most two cases:

\begin{lem}\label{l:reduction}
Either conclusion of Theorem \ref{t:main1} or~\ref{t:main2} holds for $S\cong S_{0,6}$ if and only if it holds for $S\cong S_{2,0}$.
\end{lem}
\begin{proof}
It is well--known that $\Mod(S_{2,0})$ and $\Mod(S_{0,6})$ are commensurable (see \cite{FM2011}, Theorem 9.2, for instance).  It follows that $\aga<\Mod(S_{2,0})$ if and only if $\aga<\Mod(S_{0,6})$.  It is also well--known (see \cite{RS2011}, for instance) that the curve complexes $\mC(S_{2,0})$ and $\mC(S_{0,6})$ are isomorphic (in fact, the fact that the mapping class groups are commensurable implies that the curve graphs are isomorphic; see \cite{acyl}, Lemma 3 and Proposition 4).  In particular, the two curve graphs have the same finite subgraphs.  The lemma follows immediately.
\end{proof}

We do not know how to resolve the case $\xi(S)=3$ at this time.
\bibliographystyle{amsplain}
\def\soft#1{\leavevmode\setbox0=\hbox{h}\dimen7=\ht0\advance \dimen7
  by-1ex\relax\if t#1\relax\rlap{\raise.6\dimen7
  \hbox{\kern.3ex\char'47}}#1\relax\else\if T#1\relax
  \rlap{\raise.5\dimen7\hbox{\kern1.3ex\char'47}}#1\relax \else\if
  d#1\relax\rlap{\raise.5\dimen7\hbox{\kern.9ex \char'47}}#1\relax\else\if
  D#1\relax\rlap{\raise.5\dimen7 \hbox{\kern1.4ex\char'47}}#1\relax\else\if
  l#1\relax \rlap{\raise.5\dimen7\hbox{\kern.4ex\char'47}}#1\relax \else\if
  L#1\relax\rlap{\raise.5\dimen7\hbox{\kern.7ex
  \char'47}}#1\relax\else\message{accent \string\soft \space #1 not
  defined!}#1\relax\fi\fi\fi\fi\fi\fi}
\providecommand{\bysame}{\leavevmode\hbox to3em{\hrulefill}\thinspace}
\providecommand{\MR}{\relax\ifhmode\unskip\space\fi MR }
\providecommand{\MRhref}[2]{%
  \href{http://www.ams.org/mathscinet-getitem?mr=#1}{#2}
}
\providecommand{\href}[2]{#2}

\end{document}